\documentclass[11pt]{article}

\usepackage{natbib}
\usepackage{amssymb}
\usepackage{amsmath,enumerate,rotate}

\usepackage[latin1]{inputenc}
\usepackage[english]{babel}
\usepackage{amssymb}
\usepackage{amsmath,enumerate,rotate}
\usepackage{amssymb,latexsym}
\usepackage{amsthm}
\usepackage{natbib}

\usepackage{color}
\usepackage{soul}

\theoremstyle{plain}

\newtheorem*{theorem*}{Theorem}

\theoremstyle{remark}

\newtheorem{example}{Example}

\newtheorem*{definition*}{Definition}

\newcommand\tsum{\textstyle\sum}
\newcommand\diff{\mathrm{d}}

\newcommand\ind{\mathbb{I}}

\newcommand\BR{\mathbb{R}}

\newcommand\la{\lambda}

\usepackage{amssymb}
\usepackage{amsmath}
\usepackage{latexsym}
\usepackage[mathcal]{eucal}
\usepackage{mathrsfs} %per usare il carattere mathscr
\usepackage{verbatim}
\usepackage{url}

\title{A definition of conditional probability distribution with non--stochastic information}

\author{Pier Giovanni Bissiri \footnote{Pier Giovanni Bissiri is Postdoctoral reasercher, 
Dipartimento di Statistica, Università degli Studi di Milano--Bicocca,
Edificio U7, Via Bicocca degli Arcimboldi 8, Milano 20126, Italy (e-mail: pier.bissiri@unimib.it).}$\;\;$
and Stephen G. Walker \footnote{Stephen G. Walker is Professor, SMSAS, University of Kent, Canterbury, Kent, CT2 7NZ, UK (e-mail: S.G.Walker@kent.ac.uk).}}

\begin{document}

\maketitle

\begin{abstract}
The current definition of a conditional probability distribution enables one
to update probabilities only on the basis of stochastic information.
This paper provides 
a definition for conditional probability distributions with non--stochastic information.  
The definition is derived as a solution of a decision theoretic problem,  
where the information is connected to the outcome of interest via a loss function. 
We shall show that the Kullback--Leibler divergence plays a central role.  
Some illustrations are presented.
\end{abstract}


%\begin{keyword}[class=AMS]
%\kwd{03B48}\kwd{60A99}
%\kwd[; secondary ]{62C99}
%\end{keyword}
%\begin{keyword}
%\kwd{Conditional probability distribution}\kwd{conditional probability density}
%\kwd{loss function}\kwd{Kullback--Leibler divergence}\kwd{g-divergence}
%\end{keyword}

%\end{frontmatter}
%\setcounter{section}{-1}

\section{Introduction}\label{s: intr}
The theory of conditional probability distributions
is a well-established mathematical theory that
provides a procedure to update probabilities taking into account new information.
%\hl{In the framework used in most applications, conditional distributions can be identified by means of conditional densities.}
To motivate the new work in this paper, we mention that such a procedure is available only if the information which is used to update the probability concerns stochastic events; that is, events to which a probability is assigned.
In other words, such information needs to be already included into the probability model. 
%The purpose of this paper is to provide a definition of the conditional distribution of a random variable $Y$ on the basis of information $I$ which is not related to stochastic events. 
%The statistical implications are most pertinent to Bayesian inference where non stochastic information, obtained by experts for example, are often employed. 
% For more information about this, see \cite{BissiriWalker09, BWb, BWc}.

\subsection{Notation}
Before proceeding, we introduce the notation. 
Let $Y$ be a random variable on a probability space $(\Omega, \mathscr{F},\mathbb{P})$, which will be the outcome of interest,  
and valued into a measurable space $(\mathbb{Y},\,\mathscr{Y})$ with probability distribution $P$. Hence, $P$ represents initial belief about the outcome concerning $Y$.  
By $I$ we shall denote the information obtained about $Y$.  
If $I$ is stochastic, then we shall represent it by a random variable $X$ 
from $(\Omega, \mathscr{F},\mathbb{P})$ into $(\mathbb{X},\,\mathscr{X})$ with probability 
distribution $Q$ and $I$ will be assumed to be an outcome of $X$. 
% Let $X$ and $Y$ be two random variables on a probability space $(\Omega, \mathscr{F},P)$. Denote by $(\mathbb{X},\,\mathscr{X})$ and by $(\mathbb{Y},\,\mathscr{Y})$ the range of $X$ and $Y$, respectively.
We will denote by $P_I$ the updated $P$ given information $I$.

We will let $D$ denote the  Kullback-Leibler divergence (relative entropy), i.e.
\[D(Q_1,Q_2) = \int \log\bigg(\frac{\diff Q_1}{\diff Q_2}\bigg)\ \diff Q_1,\]
for any couple $(Q_1,Q_2)$ of probability measures such that $Q_1\ll Q_2$. 
More generally we define the $g$-divergence:
\[D_g(Q_1,Q_2) = \int g\bigg(\frac{\diff Q_1}{\diff Q_2}\bigg)\ \diff Q_2,\]
for any couple $(Q_1,Q_2)$ of probability measures such that $Q_1\ll Q_2$,
where $g$ is a convex function from
$(0,\infty)$ into $\BR$ such that $g(1)=0$. This class of probability discrepancies has been introduced and studied independently by \cite{AliSilvey} and \cite{Csiszar}. 
 The Kullback--Leibler divergence is a particular case, which can be obtained taking $g(x)=x\log(x)$. 

\subsection{Mathematical framework}\label{s: math_framework}
When the standard definition of conditional probability does not apply, 
for reasons we discuss later, we present an alternative definition based on a mathematical 
decision theoretic framework. 
When information received is non--stochastic, but relevant 
to an outcome of interest, 
we cannot use a probability distribution and so we 
need an alternative way to connect the information $I$ with outcome of 
interest $Y$. 
We do so using loss functions.
 
The purpose of this paper is to provide a definition of a conditional distribution of $Y$ on the basis of $I$, which we shall denote by $P_I$. %which is not related to stochastic events. 
We take the pair $(I, P)$ to $P_I$ as the solution to a decision problem based on the 
minimization of a cumulative loss function. 
This loss function will be defined on the class of
probability measures on $\mathscr{Y}$ that are absolutely
continuous with respect to $P$, call this $\mathscr{P}$. 
Indeed, the conditional probability should be zero on every
event whose unconditional probability is zero.  
Here, $\lambda\in\mathscr{P}$ will denote the action and the best choice, i.e. 
minimizing the loss function, 
will be defined as the conditional probability distribution for $Y$ given $I$. 
In order to properly assess the loss function, 
it will be expressed as the sum (cumulative loss) of two terms, i.e. 
\begin{equation}\label{f: loss_function} 
%L(\la):=L_I(\lambda, P) = 
L(\la) = H_I(\la)+l(\lambda,P),
\end{equation}
where $l(\la, P)$ is a discrepancy between the probability measure $\la$ and $P$ and
$H_I(\la)$ is the component of the loss that takes into account the information relating to $I$.
%which treats both $I$ and $P$ as independent pieces of information about $Y$ and hence the suitability of a cumulative or additive loss function.
% So, here we have merely written down a cumulative loss with a loss for each piece of information; i.e. $I$ and $P$.
In fact, we will show that $l(\la, P)$ should be the Kullback--Leibler divergence 
for coherence purposes. 
So, $P_I$ will be defined as that $\la$ which minimizes $L(\la)$.
%solution of a decision problem. 

\subsection{Relation to the literature}
In the literature, 
definitions of conditional probability, such as the Jeffrey's Rule of conditioning, 
are given where new information is not put in terms of the occurrence of an event included in the model.   
These definitions rely on the assumption that the information can be given 
in the form of a constraint (or a combination of constraints) 
on the probability.  
% whereas our approach is based solely on the use of  loss functions. The loss function they minimize is just the Kullback--Leibler divergence $D(\la, P)$ and they consider constraints such as, for instance,
Constraints considered are of the type
\begin{equation}\label{f: constraint}
\int_{\mathbb{Y}} g(y) \la(\diff y) >0,
\end{equation} 
where $g$ is a measurable real function on $\mathbb{Y}$    
%Constraints with a not strict inequality are also considered.   
and the strict inequality is sometimes replaced by a not strict one. 
%Instead, the approach presented in this paper will be based solely on the use of  loss functions and  not on constraints imposed on the updated distribution.
The idea is to minimize $D(\la,P)$ subject to the constraint 
\eqref{f: constraint}, which represents information $I$.
This problem can be solved, i.e. $P_I$ can be 
obtained, by minimizing $D(\la,P)$ subject to the constraint 
using Lagrange multipliers.

Such a procedure of condizionalization is a specific case in our approach. 
In fact, it is equivalent to minimize the loss function \eqref{f: loss_function} taking 
$l$ equal to the Kullback--Leibler divergence and 
\[
H_I(\la)=
\begin{cases} 
0 &\text{if}   \int_{\mathbb{Y}} g(y) \la(\diff y) >0,\\
+ \infty &\text{if}  \int_{\mathbb{Y}} g(y) \la(\diff y) \leq 0.
\end{cases}
\]

For more details about conditionalization based upon constraints on the 
conditional distribution, see 
\cite{frassen},  
\cite{skyrms}, \cite{domotor}, 
\cite{diaconis} and \cite{shore}. 
Our approach is different as we %can deal with more types of information. 
%Indeed, our definition can
encompass potentially arbitrary information 
about $Y$, so as long as it is possible to construct a loss function $h_I(y)$ 
for each $Y$ given $I$. 
%\subsection{Philosophy}

\subsection{Motivation}
%For the sake of clarity, and to make the notation concrete, 
The random variable $Y$ represents an unknown quantity to which a probability distribution has been assigned and needs to be updated on the basis of new information $I$.
If $I$ coincides with an outcome of another random variable $X$,
then it is possible to update the unconditional distribution of $Y$ to the probability distribution of $Y$ given $X$.
However, to do this, it is required to know
all the possible alternatives of $I$, that is, all the outcomes of $X$.
Moreover, it is required to assess the joint distribution of $X$ and $Y$ or the conditional distribution of $X$ given $Y$. This is quite easy if, for instance, $I$ is  known to be an outcome of some well-defined random experiment.
In many situations, one has seen the outcome $X$ and in order to establish an update of the distribution of $Y$, one needs to retrospectively ponder and imagine a joint probability model.

This difficulty arises in different puzzles such as, for instance,
Freund's puzzle of the two aces, introduced by \cite{Freund65}.
For other puzzles about conditional probabilities, see, for instance, \citet{Gardner59}.

These puzzles have been widely used to discuss the concept of conditional probability.
\cite{Hutchison99, Hutchison08}
emphasizes that the updating process needs to take into account
the circumstances under which the truth of $I$ was conveyed.
Also, \cite{Bar-Hillel82} claim that to know how the knowledge was obtained is
``a crucial ingredient to select the appropriate model".
These scholars present different views about the concept of conditionalization, but
all agree on the fact that there would not be a problem if it was known how the information $I$ became available, and therefore one could build a model including $I$.

The concept of conditional probability distributions
is certainly appropriate as a procedure to update probabilities on the basis of any new information
that was already included in the probability model.
But it can be difficult to construct a model that considers all possible relevant information that in the future could become available.
Therefore, the problem arises when one obtains some new and possibly unexpected information
and wants to use it to update a probability distribution.
Indeed, it does not seem appropriate to assess the probability of something which has been already observed.
%\underline{In the framework used in most applications, conditional distributions can be identified by means of conditional densities. In this paper, the concept of conditional density is generalized. In this way, the concept of conditional probability distribution is extended so that it is possible to update probabilities on the basis of information which is not related to stochastic events.}
% The purpose of this paper is to show how to update the distribution of a random variable $Y$ on the basis of information $I$ which is not related to stochastic events. 
Our basic assumption is that the information $I$ can be connected to the outcome of interest 
via a loss function $H_I$ defined on the set of all possible outcomes of $Y$. 
The conditional distribution of $Y$ given $I$ 
will be defined as the one that minimizes 
a cumulative loss in the form given by \eqref{f: loss_function}. 
%obtained as sum of $H_I$, which takes into account $I$, and a discrepancy $l$ with the unconditional distribution $P$ of $Y$.  
In this way, it is possible to update the distribution of $Y$, 
even if $I$ is some new unexpected information, which was not included in the probabilistic framework. 
It will be shown that if instead $I$ is the outcome of a random variable $X$ and there is a joint density $f$ for $(X, Y)$, then one can recover as particular case the conditional distribution of $Y$ given $X$. To do this, $l$ is taken to be the Kullback--Leibler divergence. 
It will be proved that in general it is necessary for the updating procedure to be coherent that $l$ is the Kullback--Leibler divergence. 

%Clearly, we are dealing with information about outcomes of interest which take many forms. It is well--known that Bayesian inference can rely on expert opinions, which often takes the form of non stochastic information. In particular, our framework will allow the coherent update for probabilities involving such information and hence at least the practical relevance of our framework to Bayesian inference.

\subsection{Description of the paper}
%This %extension
%will be presented in detail in Section \ref{s: main}. 
Section \ref{s: main} contains the main results. 
In Section \ref{s: example}, some examples will be considered. One such is as follows:
assume that $Y$ is a scalar quantity and
one learns that $Y$ is close to zero. An answer will be given to this question:
how could one update the distribution of $Y$ \emph{after} learning such information? 
Section \ref{s: discussion} contains a discussion. 

\section{Defining conditional probability distributions with non--stochastic information}
%Extending the definition}
\label{s: main}
This section reports the current definition of conditional probability distribution 
and presents and motivates our definition for conditional probability distribution with non--stochastic information.
% Let $X$ and $Y$ be two random variables on a probability space $(\Omega, \mathscr{F},P)$. Denote by $(\mathbb{X},\,\mathscr{X})$ and by $(\mathbb{Y},\,\mathscr{Y})$ the range of $X$ and $Y$, respectively.

\subsection{The current definition}
In probability theory,
a conditional distribution of $Y$ given $X$ is a map $p$ 
from $\mathscr{Y}\times \mathbb{X}$ into $\BR$ such that:
\begin{itemize}
\item for each $x$ in $\mathbb{X}$, $p(\cdot,x)$ is a probability measure on $\mathscr{Y}$,
\item for each  $B$ in $\mathscr{Y}$, $p(B,X(\omega))$ is a version of the conditional probability
$\mathbb{P}(Y\in B\mid X(\omega))$, i.e.
for each $A$ in $\mathscr{X}$ and each $B$ in $\mathscr{Y}$,
\begin{equation}\label{f: def.}
\mathbb{P}\{X\in A,\, Y\in B\} = \int_A p(B,x)\, \diff Q(x),\end{equation}
where $Q$ denotes the probability distribution of $X$.
\end{itemize}
%This definition can be found, for instance, in \cite{Billingsley}.
The conditional distribution is known to be essentially unique, i.e. unique only up to a.s. equality.
This is a consequence of $X$ being stochastic. In fact, as \citet[page 160]{Feller} points out, if, for instance, the distribution of $X$ is concentrated on a subset $\mathbb{X}_0$ of $\mathbb{X}$, no natural definition of $p(B,x)$ is possible for $x$ outside $\mathbb{X}_0$. Nevertheless, in individual cases, there usually exists a natural
choice dictated by regularity requirements.

Moreover, it is well known that
conditional distributions do not always exist unless some conditions are satisfied by the spaces $(\mathbb{X}, \mathscr{X})$ and $(\mathbb{Y}, \mathscr{Y})$. For more information about conditional probability distributions, see, for instance, \cite{Feller} or \cite{Billingsley}.
%, as shown for the first time by \cite{Dieudonné}.J. Dieudonn´e - Sur le th´eor`eme de Lebesgue-Nikodym III, Ann. Univ. Grenoble, Vol. 23, pp. 25-53, (1948).

%Moreover, assume
This paper will consider the case in which
 there are two $\sigma$-finite measures $\mu$ and $\nu$ on $\mathscr{F}$ such that
the probability distribution of $(X,\,Y)$ is absolutely continuous with respect to $\mu\times \nu$.
Denote its density by $f$.
This is a general framework which includes most applications and enables to find easily
an expression for the conditional distributions.
Generally, $\mathbb{X}$ and
$\mathbb{Y}$ are subsets of $\BR^k$, for some $k$, and $\mu$ and $\nu$ are the corresponding
Lebesgue measure.

%Denote by $f$ the density of the probability distribution of $(X,\,Y)$ with respect to $\mu\times\nu$. Then, for every $B$ in $\mathbb{Y}$, a version of the conditional probability $P(Y\in B \mid X)$ is equal to $q(B, X)$, where
If $f$ is  the density of the probability distribution of $(X,\,Y)$ with respect to $\mu\times\nu$, then
one can take
\begin{equation}\label{f: prob_cond}
p(B,x)=\frac{\int_B f(x,y)\ \nu(\diff y)}{\int_\mathbb{Y} f(x,y)\ \nu(\diff y)},
\end{equation}
for every $B$ in $\mathbb{Y}$
and every $x$ in $\mathbb{X}$ such that
\begin{equation}\label{f: marginalx}
0\,<\,f_X(x):=\int_\mathbb{Y} f(x,y)\ \nu(\diff y)\,<\,\infty.\end{equation}
%Clearly, \eqref{f: prob_cond}, as function of $B$, is a probability measure and therefore  a conditional distribution for $Y$ given $X$.
Note that $p(\cdot,x)$
is absolutely continuous w.r.t. $\nu$ and its density is
\begin{equation}\label{f: cond_density}
f_{Y|X}(y|x):= f(x,y)/f_X(x),
\end{equation}
%where \[f_X(x):=\int_{\mathbb{X}} f(x,y) \ \nu(\diff y).\]
for every $x$ in $\mathbb{X}$ satisfying \eqref{f: marginalx}.
The density \eqref{f: cond_density}, which is called the conditional density of $Y$ given $X$,
is what is used in most application to find an expression for the conditional distribution.
Therefore,
\eqref{f: prob_cond} deserves to be considered as the ``practical definition" of
conditional probability distribution.
%In other words, there is a conditional distribution for $Y$ given $X$.
Indeed, it is the natural version of the conditional distribution of $Y$ given $X$ whenever a joint density $f$ exists for $X$ and $Y$.

\subsection{The loss function}%Updating probabilities via a loss function}
Given it is not always possible to relate new information $I$ to $Y$ through probability models,
instead, we will rely on the use of loss functions to ``connect" the information $I$ to $Y$. We will deal with the theory first, and then present some examples.

Before proceeding, let us recall that
%It is known that
$q(B,\cdot)$ %, the conditional distribution of $Y$ given a random variable $X$
satisfying \eqref{f: def.}
can be seen as the solution of a minimization problem whenever $Y$ is in $L^2(\Omega, \mathscr{F},P)$, by resorting to the theory of Hilbert spaces \citep[see, for instance,][]{JacodProtter}.
Clearly, this approach relies on the joint distribution of $X$ and $Y$ and therefore is not available when $X$ is replaced by some non--stochastic information $I$.
%In fact, resorting to the theory of Hilbert spaces, it is clear that $q(B,\cdot)$ satisfies \eqref{f: def.} if and only if the functional $\mathbb{E}(\ind_{\{Y\in B\}}-p(X))^2$, where $p: \mathbb{X}\to \BR$, reaches its minimum at $p(\cdot)=q(B, \cdot)$.

So, our aim is to define a conditional probability distribution as a solution of a decision problem with a fully motivated loss function; connecting the action, i.e. the conditional distribution, with current and given pieces of information: namely
the probability distribution $P$ of $Y$ and $I$, respectively. 

The form of the loss function we consider is \eqref{f: loss_function}. 
In particular, $H_I(\la)$ will be taken in the integral form i.e. the average or expected loss
\begin{equation*}%\label{f: H}
H_I(\la) = \int_{\mathbb{Y}} h_I(y)\; \lambda(\diff y),
\end{equation*}
where $h_I(\cdot, P)$ is a loss function defined on $\mathbb{Y}$. It is more reasonable to assess the loss relating to $Y$ and therefore it is reasonable to be able to construct 
$h_I(y)$. 
Examples will be considered later. If $\lambda$ then represents beliefs about $Y$, it is appropriate to consider the expected loss here.
Therefore, to define
%the updating procedure,
conditional distributions,
a cumulative loss will be used of the following form:
\begin{equation}\label{f: gen.loss}
\int_{\mathbb{Y}}
h_I(y)
\ \la(\diff y)\ +\, \ l(\la, P).
\end{equation}
This general cumulative loss then represents or assesses the loss to the decision maker if they select probability measure $\lambda$ in the presence of information $I$ and $P$. 
% Note that one can take Lagrangian

%\vspace{0.2in} \noindent {\bf 2.1 Stochastic information.}
\subsection{Stochastic information}
Let us see how this works when indeed $I$ is equivalent to a random variable $X$
and there is a joint density $f$ for $(X, Y)$.
In this setting, the conditional distribution \eqref{f: prob_cond}
 arises as the solution
of a decision theoretic problem. To see this, for every $x$ in $\mathbb{X}$
satisfying \eqref{f: marginalx},
define the following loss function $\bar{L}_x$:
\begin{equation}\label{f: loss1}\bar{L}_x(\la)\,:= \,
-\int_{S} \log (f(x,y)/f_Y(y))\ \la(\diff y)\ +\, \ D(\la, P),
\end{equation}
where
\begin{equation*}
f_Y(y):=\int_{\mathbb{X}} f(x,y) \ \mu(\diff x),
\end{equation*}
 $S$ is the set of all $y$ in $S$ such that $0<f_Y(y)<\infty$, 
$P$ is the probability distribution of $Y$,
$\la$ is a probability measure on $\mathbb{Y}$ absolutely continuous
w.r.t. $P$, and $D$ %denotes the  Kullback-Leibler divergence (relative entropy), i.e. \[D(Q_1,Q_2) = \int \log\bigg(\frac{\diff Q_1}{\diff Q_2}\bigg)\ \diff Q_1,\] for any couple $(Q_1,Q_2)$ of probability measures such that $Q_1\ll Q_2$.
%So we are using the loss functions
The loss \eqref{f: loss1} is of the form \eqref{f: gen.loss} with
$$l(\lambda,P)=D(\la, P),$$
%$$H(\la)=H(\lambda,X):=-\int_{S} \log (f(X,y)/f_Y(y))\ \la(\diff y)$$
and
\begin{equation}\label{f: self-information-loss}\begin{split}
h_I(y)=h(y, x):&=-\ind_S(y)\log (f(x,y)/f_Y(y))\\&=-\ind_S(y)\log f_{X| Y}(x|y),
\end{split}\end{equation}
where $\ind_S(y)$ is equal to $1$ or $0$ depending on whether $y$ belongs to $S$ or not.

For every $x$ in $\mathbb{X}$ satisfying \eqref{f: marginalx},
the conditional distribution $p(\cdot,x)$ given by \eqref{f: prob_cond} %of $Y$ given $X$
minimizes the loss $\bar{L}_{x}$, since
\[\bar{L}_x(\la) =
D(\la, p(\cdot,x)) - \log\bigg(\int_{\mathbb{Y}}\, f(x,y)\,  \nu(\diff y)\bigg).\]

In the loss \eqref{f: loss1}, the first addendum
depends on the joint density function of $X$ and $Y$ and therefore,
to be able to define such loss,
$X$ needs to be stochastic. In other words, a probability distribution has to be assigned to $X$.

The loss \eqref{f: self-information-loss} %used here
is known as the self--information loss function and the most commonly used when $x$ has come from a specified family of densities. %So the loss can be viewed as $$\int_S h(y,x)\,\lambda (\diff y)$$ where $$h(y,x)=-\log (f(x,y)/f_Y(y))=-\log f_{X|Y}(x|y).$$
%That is, the expected or average loss, using the self--information loss function $-\log f_{X|Y}(x|y)$.
So, $H_I$ %in \eqref{f: H}
 turns out to be the the expected or average loss, using the self--information loss function $-\log f_{X|Y}(x|y)$.

%\vspace{0.2in} \noindent {\bf 2.2 Non--stochastic information.}
\subsection{Non--stochastic information}
If the random variable $X$ is replaced by some non--stochastic information $I$,
%then it is reasonable to replace the loss function with another one, say $h(y)$. As usual, $h(y)$ evaluates the additional loss in outcome $y$ due to the acquirement of $I$.
then the self--information loss \eqref{f: self-information-loss} cannot be defined, but
one can still resort to a loss function of the form \eqref{f: gen.loss}, assessing $h_I(y)$ in a
different way.
As usual, $h_I(y)$ evaluates the additional loss in outcome $y$ 
due to the acquirement of $I$.
Some examples for this will be considered later. 

In the loss \eqref{f: loss1}, the Kullback--Leibler divergence from the marginal of $Y$ can also be replaced by a more general discrepancy,
such as the $g$-divergence. 
This leads us to consider a more general loss function than \eqref{f: loss1} as follows:
\begin{equation}\label{f: loss}
\int_{\mathbb{Y}}
h_I(y)
\ \la(\diff y)\ +\, \ D_g(\la, P),
\end{equation}
where $h_I$ is assessed after learning
%the random variable $X$ has been replaced by
$I$,
information which does not need to be stochastic.
As the loss \eqref{f: loss1}, the loss \eqref{f: loss} is defined on the class of
probability measures on $\mathscr{Y}$ that are absolutely
continuous with respect to $P$, which is reasonable. 
%\begin{definition}
% If there is a unique  probability measure in such class  that minimizes \eqref{f: loss}
Assume there is a unique probability measure that minimizes \eqref{f: loss}
in the class of probability measures on $\mathscr{Y}$ absolutely
continuous with respect to $P$. 
Then, it will be called 
the conditional distribution of $Y$ given the information $I$  
(according to the discrepancy $D_g$ and the loss $h_I$)   
and it will be denoted by $P_I$. 
%\end{definition}

%\underline{At this stage, it is worth introducing some further notation and some definition. Given two pieces of information $I_1$ and $I_2$, $I_1 I_2$ (or equivalently $I_2 I_1$) will denote the information available after $I_1$ and $I_2$ have been both acquired. Moreover, define $h_{I_2\mid I_1}=h_{I_1 I_2}-h_{I_1}$, where now $h_I$ is short for the loss function based on information $I$. The loss $h_{I_2\mid I_1}$ is the additional loss due to the acquirement of $I_2$ in presence of $I_1$. It can be greater or bigger than $h_{I_2}$ according to the way in which the two pieces of information interact in the decision process when they are both available. If in particularly $h_{I_2\mid I_1}=h_{I_2}$, or equivalently $h_{I_1 I_2}=h_{I_1}+h_{I_2}$, then $I_1$ and $I_2$ will be said to be independent.}

%\underline{Now, consider again the loss \eqref{f: loss} and denote by $I_0$ the initial information (i.e. the available information before acquiring $I$), which in terms of probability distributions is translated into $P$. The loss function $h_I$ in \eqref{f: loss} is the loss produced by the new information $I$ according to a subject who already has information $I_0$. Therefore, $h_I$ should be based on $I$, but also needs to take into account the interaction between $I$ and $I_0$. Clearly, $h_I$ has to be assessed to be equal to $h_{I\mid I_0}$.}

At this stage, assume that another piece $J$ of information 
is available in addition to $I$ 
and that $I$ and $J$ are not overlapping pieces of information. 
% In the stochastic case, this means that $I$ and $J$ are outcomes of two independent random variables.
This happens, for instance, in the stochastic case when $I$ and $J$ are outcomes of two independent random variables. 
We shall write $I J$ (or equivalently $J I$)
to denote the information obtained combining $I$ with $J$. 
%We choose $h_I$, $h_J$ and $h_{IJ}$ satisfying the following additivity property: \begin{equation}\label{f: h}\begin{split} h_{IJ}(y, P)&=h_I(y)+h_J(y, P_I)\\ &=h_J(y, P)+h_I(y, P_J). \end{split} \end{equation} Indeed, $h_{I}(y, P)$ is the additional loss due to the acquirement of $I$ in presence of the information contained in $P$.  In analogous way, $h_J(y, P_I)$ is the additional loss due to the acquirement of $J$ in presence of the information contained in $P_I$, which is obtained combining the information given by  $P$ and $I$. As a consequence, the additional loss due to the joint acquirement of $I$ and $J$ has to be equal to the sum of the two losses, as in \eqref{f: h}.
Being $I$ and $J$ not overlapping, we choose $h_I$, $h_J$ and $h_{IJ}$ satisfying the following additivity property:
\begin{equation}\label{f: h}
 h_{IJ}(y) =h_I(y)+h_J(y). 
  \end{equation}

Clearly, updating the distribution $P$ on the basis of $I$ and $J$ and updating
the conditional distribution $P_I$ on the basis of $J$ only, should yield the
same probability distribution for $Y$. 
In the first case, the updated probability distribution is obtained by minimizing the loss:
\begin{equation}\label{f: loss_1}
\int_{\mathbb{Y}}
h_{IJ}(y)%h_{IJ}(y,P)
\ \la(\diff y)\ +\, \ D_g(\la, P).
\end{equation}
 In the second one, the loss to minimize is: %it is
 \begin{equation}\label{f: loss_2}
\int_{\mathbb{Y}}
h_{J}(y) %h_{J}(y,P_I)
\ \la(\diff y)\ +\, \ D_g(\la, P_I).
\end{equation}
The two losses \eqref{f: loss_1} and \eqref{f: loss_2} should yield the same updated probability distribution for $Y$.

For this coherence condition to be in force, it is necessary %(and sufficient)
that the discrepancy $D_g$ is the Kullback-Leibler divergence.
To be more precise, the following theorem can be stated:

\begin{theorem*}
Let $\bar{P}:=P_I$, and 
%assume that:
assume that \eqref{f: h} holds and 
\begin{equation}\label{f: coherence}
P_{I J} = \bar{P}_J,\end{equation}
for every probability measure $P$ on $\mathscr{Y}$ and for every choice of the loss functions 
%$h_{I\mid I_0}$ and $h_{J\mid I_0 I}$ 
$h_{I}$ and $h_{J}$ 
such that $P_I$, 
$P_{IJ}$ and $\bar{P}_J$ are all properly defined.

Then $D_g$ is the Kullback-Leibler divergence.
\end{theorem*}

%\vspace{0.2in}

\begin{proof}
This result is proven from a different starting point in \citet[Theorem 2.5]{BissiriWalker09}.
Here, a shorter proof is given %by considering also the differentiability of $g$.
by assuming the differentiability of $g$.

Assume that $\mathbb{Y}$ contains at least two distinct points, say $y_0$ and $y_1$. Otherwise, $P$ is degenerate and the thesis is trivially satisfied.

To prove this theorem, it is sufficient to consider a very specific choice for $P$, taking
$P=p_0\delta_{y_0}+ (1-p_0)\delta_{y_1},$
where $0< p_0< 1$. 
%To simplify the notation, we shall write $h_I(y, p)$ instead of $h_I(y, \la)$ if $\la=p\delta_{y_0}+ (1-p)\delta_{y_1}$. 
Any probability measure $\lambda \ll P$ has to be equal to
$p\delta_{y_0}+(1-p)\delta_{y_1}$, for some $0\leq p\leq 1$. 
Therefore, in this specific situation, %if the function $h_I(y, p_0)$ is replaced by $h(y)$ on  $\mathbb{Y}$, then 
the loss \eqref{f: loss} becomes:
\begin{equation*}
\begin{split}
%l(p;p_0;h_I(\cdot, p_0))&:=p\,h_I(y_0, p_0)\, +\,(1-p)\,h_I(y_1, p_0))\\
l(p,p_0,h_I)&:=p\,h_I(y_0)\, +\,(1-p)\,h_I(y_1)\\
&\phantom{=}+\, p_0\,g\left(\frac{p}{p_0}\right)\,+\,(1-p_0)\,g\left(\frac{1-p}{1-p_0}\right).
\end{split}
\end{equation*}
Denote by $p_1$ the probability $P_I(\{y_0\})$, i.e. the minimum point of $l(p,p_0,h_I)$ %$l(p,p_0,h_I(\cdot, p_0))$
as a function of $p$, and by $p_2$ the probability $P_{IJ}(\{y_0\})$. 
By hypotheses, $p_2$ is the unique minimum point of both loss functions
% $l(p,p_1,h_{J}(\cdot, p_1))$ and $l(p;p_0;h_{I J}(\cdot, p_0))$. 
$l(p,p_1,h_{J})$ and $l(p,p_0,h_{I J})$.  
Again by hypothesis, we shall consider only those functions 
% $h_I(\cdot,p_0)$ and $h_J(\cdot,p_1)$ such that 
$h_I$ and $h_J$
 such that each one of the functions 
% $l(p;p_0;h_I(\cdot,p_0))$, $l(p;p_1;h_J(\cdot,p_0))$, and $l(p,p_0,h_{IJ}(\cdot,p_0))$, 
$l(p,p_0,h_I)$, $l(p,p_1,h_J)$, and $l(p,p_0,h_{IJ})$, 
as a function of $p$, has a unique minimum point, which is $p_1$ for the first one and $p_2$ 
for the second and third one.
%By hypothesis, we shall consider only those functions $h$ such that $l(p;p_0;h)$,  as a function of $p$, have a unique minimum point, say $p^*$.
% Such $p^*$ has to be strictly bigger than zero and strictly smaller than one: this was proved by \citet*[Lemma~3.1]{BissiriWalker09}.
The values $p_1$ and $p_2$ have to be strictly bigger than zero and strictly smaller than one: this was proved by \citet*[Lemma~2]{BissiriWalker09}.
%Hence, $p^*$ has to be a stationary point of $l$, i.e. \begin{equation}\label{f: loss.} \, g'\left(\frac{p^*}{p_0}\right)\,-\,g'\left(\frac{1-p^*}{1-p_0}\right)\,=\,h(y_1)\, -\, h(y_0).\end{equation}
%%%%%Being $g$ convex, $l(p;p_0;h)$ is convex as well as function of $p$ and therefore \eqref{f: loss.} is sufficient (and necessary) for $p^*$ to be the minimum point. Hence, it is clear that there is a function $h$ such that a minimum $p^*$
%At this stage, denote by $p_1$ the minimum point of $l(p;p_0;h_{I}(\cdot, p_0))$as a function of $p$.Hence by \eqref{f: loss.},
Hence, $p_1$ has to be a stationary point of % $l(p;p_0;h_{I}(\cdot, p_0))$ 
$l(p,p_0,h_{I})$
and $p_2$ of both the functions 
% $l(p;p_1;h_{J}(\cdot, p_1))$ and $l(p;p_0;h_{I J}(\cdot, p_0))$.
 $l(p,p_1,h_{J})$ and $l(p,p_0,h_{I J})$. 
Therefore, 
%\begin{equation}
\begin{align}\label{f: loss.1}
g'\left(\frac{p_1}{p_0}\right)\,-\,g'\left(\frac{1-p_1}{1-p_0}\right)\,&=\,
% h_{I}(y_1,p_0)\, -\, h_{I}(y_0,p_0),\\
h_{I}(y_1)\, -\, h_{I}(y_0),\\
%\end{equation} Hence, one can apply \eqref{f: loss.} twice, obtaining that
\label{f: loss.2}
\, g'\left(\frac{p_2}{p_0}\right)\,-\,g'\left(\frac{1-p_2}{1-p_0}\right)\,&=\,
% h_{I J}(y_1, p_0)\, -\, h_{I J}(y_0, p_0),\\
h_{I J}(y_1)\, -\, h_{I J}(y_0),\\
\label{f: loss.3}
\, g'\left(\frac{p_2}{p_1}\right)\,-\,g'\left(\frac{1-p_2}{1-p_1}\right)\,&=\,
%h_{J}(y_1, p_1)\, -\, h_{J}(y_0, p_1).
h_{J}(y_1)\, -\, h_{J}(y_0).
\end{align}
Recall that 
% $h_{I J}(\cdot, p_0)=h_{J}(\cdot, p_1)+h_{I}(\cdot, p_0)$ by \eqref{f: h}.
$h_{I J}=h_{J}+h_{I}$ by \eqref{f: h}.
Therefore,
summing up term by term \eqref{f: loss.1} and \eqref{f: loss.3},
and considering \eqref{f: loss.2}, one obtains:
\begin{equation}\label{f: equation}\begin{split}
g'&\left(\frac{p_2}{p_0}\right)\,-\,g'\left(\frac{1-p_2}{1-p_0}\right)\\
&\phantom{XXX}=\,g'\left(\frac{p_1}{p_0}\right)\,-\,g'\left(\frac{1-p_1}{1-p_0}\right)\, +\,
g'\left(\frac{p_2}{p_1}\right)\,-\,g'\left(\frac{1-p_2}{1-p_1}\right).
\end{split}\end{equation}

% At this stage, note that for every $p$ in $(0,1)$, one can find a function $h_I(\cdot,p_0)$ such that the value $p_1$ satisfying \eqref{f: loss.1} coincides with $p$. In the same way, being $p_1$ fixed, one can choose $h_J(\cdot,p_1)$ so that $p_2$ is equal to any fixed value in $(0,1)$. 
%%%%%a function $h_I$ such that $p$ coincides with $p^*$ that satisfies \eqref{f: loss.}. 
Recall that by hypothesis \eqref{f: loss.1}--\eqref{f: loss.3} need to hold for every two functions %$h_I(\cdot, p_0)$ and $h_J(\cdot, p_1)$ 
$h_I$ and $h_J$
arbitrarily chosen with the only requirement that $p_1$ and $p_2$ uniquely exist. 
Hence, \eqref{f: equation} needs to hold for every $(p_0,p_1,p_2)$ in $(0,1)^3$.
By substituting $t=p_0$, $x=p_1/p_0$ and $y=p_2/p_1$, \eqref{f: equation} becomes
\begin{equation}\label{f: equation.}\begin{split}
g'&\left(xy\right)\,-\,g'\left(\frac{1-txy}{1-t}\right)\\
&\phantom{XXX}=\,g'(x)\,-\,g'\left(\frac{1-tx}{1-t}\right)\, +\,
g'\left(y\right)\,-\,g'\left(\frac{1-txy}{1-tx}\right),
\end{split}\end{equation}
which holds for every $0<t<1$, and every $x,y>0$ such that $x<1/t$ and $y<1/(xt)$.
Being $g$ convex and differentiable, its derivative $g'$ is continuous. Therefore,
letting $t$ go to zero, \eqref{f: equation.}  implies that
\begin{equation}\label{f: equation+}
g'\left(xy\right)
=\,g'(x)\, +\,
g'\left(y\right)\,-\,g'(1)
\end{equation}
holds true for every $x,y>0$. Define the function
$\varphi(\cdot)=g'(\cdot) -g'(1)$.
This function is continuous, being $g'$ such,
and by \eqref{f: equation+},
$\varphi(xy)=\varphi(x)\, +\,\varphi(y)$ holds for every $x,y>0$. Hence, $\varphi(\cdot)$ is $k\ln(\cdot)$ for some $k$, and therefore
\begin{equation}\label{f: log}
 g'(x)\
=\ k\,\ln (x)\ +\ g'(1),\end{equation}
where
\(k\,=\,(g'(2)\,   -\, g'(1))/\ln(2)\).
Being $g$ convex, $g'$ is not decreasing and therefore
$k \geq 0$. If $k=0$, then $g'$ is constant, which is impossible, otherwise, for any $h_I$,
$p_1$ satisfying \eqref{f: loss.1} either would not exist or would not be unique. Therefore, $k$ must be positive.
Being $g(1)=0$ by assumption, \eqref{f: log} implies that
$g(x)\:=\:k\, x\ln(x)\, +\, (g'(1)-k) (x-1)$.
Hence,
\[    D_g(Q_1,Q_2)=
 k\int \ln \bigg(\frac{\diff Q_1}{\diff Q_2}\bigg)\ \diff Q_1 \]
holds true for some $k>0$ and
for every couple of measures $(Q_1, Q_2)$ such that $Q_1\ll Q_2$.

\end{proof}

In virtue of this theorem, the conditional distribution of $Y$ given the information $I$ is coherent
%if and
 only if it minimizes the loss
\begin{equation}\label{f: loss:}
\bar{L}(\la)\,:= \,
\int_{\mathbb{Y}}
h_I(y)
\ \la(\diff y)\ +\,  k\,\int\ln\left(\frac{\diff\la}{\diff P}\right)\diff \la,
\end{equation}
where $k$ is some positive constant.
To define the loss \eqref{f: loss:}, one needs to assess 
$h_I$
 and $k$. Notice that
a probability distribution that minimizes $\bar{L}(\la)$ , or equivalently $\bar{L}(\la)/k$,
 is uniquely identified by 
 $h_{I}/k$. 
In other words, assessing 
$h_{I}=h_0$
and $k=k_0$ is equivalent to
assess $h_{I}=h_0/k_0$ and $k=1$. 
For this reason, from now on, it will be convenient to fix $k=1$.

%In order to ensure coherence, in what follows,
In what follows, only coherent conditional distributions will be considered. Therefore,
$D_g$ will always be assessed to be the
Kullback--Leibler divergence.
Whenever a probability measure that minimizes
\eqref{f: loss:} (with $k=1$) exists and is unique, it will be called the conditional probability distribution of $Y$ given $I$ and will be denoted by $P_I$.

If
\begin{equation}\label{f: cond_int}
\int_{\mathbb{Y}} 
e^{-h_I(y)}
\, P(\diff y)<\infty,
\end{equation}
then
%the coherent conditional distribution
$P_I$ is properly defined and is equal to
\begin{equation}\label{f: Result}
P_I(A)=
\frac{\int_A 
e^{-\,h_I(y)}  
P(\diff y)}
{\int_{\mathbb{Y}} 
e^{-\,h_I(u)}
\, P(\diff u)},
\end{equation}
for every measurable subset $A$ of $\mathbb{Y}$.
In fact,
\[\bar{L}(\la) =
D(\la, P_I)
- \ln\bigg(\int_{\mathbb{Y}}\,
e^{-h_I(y)}
\,  P(\diff y)\bigg)\]
holds true for every probability measure $\la$ on $\mathscr{Y}$ such that $\la \ll P$.

By \eqref{f: Result}, it is clear that %, under \eqref{f: cond_int},
the choice of the Kullback--Leibler divergence for $D_g$ and of a loss 
$h_I$ 
satisfying \eqref{f: cond_int} is sufficient %(as well as necessary)
for the coherence condition \eqref{f: coherence}. Moreover, notice that $P_I$ is defined to be a unique probability measure, not just essentially unique.
%Notice that, under \eqref{f: cond_int}, $P(\,\cdot\mid I)$ turns out to be a unique probability measure, not just essentially unique.

\section{Illustrations}\label{s: example}
The loss function  $h_I$ 
is chosen by the decision-maker on the basis of the available information. 
Such information sometimes happens to be stochastic, i.e to belong to a set of outcomes to which a probability is assigned. If this is the case, one should update the probability distribution of $Y$ by means of the usual conditional distribution.
%\hl{In particular, when there exists a conditional density for $X$ given $Y$, this is tantamount to use} %%%%%effectively using
Whenever %it is possible to use the ``practical definition" of conditional distribution \eqref{f: prob_cond},
there is a joint density $f$ for $X$ and $Y$,
this is tantamount to use
the self--information loss function $h_I(y,x)=-\ln f_{X|Y}(x|y)$. If %it is not,
 the available information is not stochastic, then
 one can resort to the approach described in the present paper, properly assessing the loss function $h_I$. 
To see a practical and simple example, consider the situation mentioned in the Introduction:

%\begin{enumerate}

%\item
\begin{example}
$Y$ is a scalar quantity and the information $I$ is that $Y$ is close to zero. In this case,
it is natural to assess:
\[h_I(y)\,=\,w\,y^2,\]
where $w$ is some positive constant, and the conditional distribution of $Y$ given $I$ is
\[
%P(Y\in A\mid I)
P_I(A)
 = \frac{\int_A e^{-w\,y^2}\,P(\diff y)}{
\int_\mathbb{Y} e^{-w\,y^2}\,P(\diff y)}.
\]
\end{example}
%\item
\begin{example}
While for the second example everyone would know how to deal with, there is currently no formal mathematical mechanism for pursuing a conditional update. So suppose it becomes known that $Y$ belongs to $B$, for some set $B$. Not because of some preliminary random experiment but rather due to it becoming aware to the decision maker that actually $B$ is the set of possible values that $Y$ can take.
So the information is non--stochastic.
    The most natural choice is
    $$         h_I(y)=
    \left\{\begin{array}{ll} 0 & y\in B \\ \\
                                    +\infty & y\notin B\end{array}\right.$$
    from which it is easy to deduce that the $\lambda$ minimising
    $\int h_I(y)\,\lambda(\diff y)+D(\lambda,P)$
    is given by
    $$P_I(A) = \int_{A\cap B} P(\diff y)/P(B).$$

This example is relevant to a number of so-called paradoxes
whereby it becomes apparent to the decision maker that the outcome space
is smaller than the support of $P$ (e.g. Freund's paradox of the two aces). %has shrunk.
How this is learnt is crucial.
This has been pointed out by \cite{Hutchison99, Hutchison08}.
%Hutchinson could be mentioned.
If the information that $Y$ belongs to $B$ is based on some preliminary random experiment, for which
a probability model is given, then obviously the unconditional distribution of $Y$ can be updated
resorting to the current definition of conditional probability.
%to the probability distribution of $Y$ given $X$
%it is obvious %update the unconditional distribution of $Y$ to the probability distribution of $Y$ given $X$.
If not, there is not currently a rigorous justification for the usage of the conditional probability.
%If not, i.e. , what to do?
%No current rigor.
The present paper provides a formal and broad enough framework  to cover this case.
Many philosophers of science, that are mentioned in the Introduction, %Section \ref{s: intr},
have discovered paradoxes based on such scenarios.

%One of them is Freund's aces paradox: \dots
\end{example}

%\item
\begin{example}
To conclude, let us consider a simple and very concrete example.
Consider a horse race, in which six horses participate.
In order to decide how to bet, one assesses the probability for each horse to win.
Denote by $p_j$ the probability that the horse number $j$ wins, for $j\in\{1,\dotsc,6\}$.
In this example, $Y$ is the number corresponding to the horse that will win.

Before the race begins, it starts raining.
Since conditions have changed, the probabilities need to be updated.
It is problematic to pursue this aim by resorting to the current definition of conditional probability.
In fact, this requires to know the probability that it rains and that the horse number $j$ wins.
As an alternative, one could calculate the conditional probabilities of victory for each horse by applying
%As an alternative, one could apply
Bayes' theorem, which requires the probability that it rains given the victory of horse $j$. But it is raining and the race is not yet run!

It is therefore appropriate to resort to the definition of a conditional probability distribution given in this paper. To this aim, one can assess a score to evaluate the disadvantage due to the rain for each horse. Denote by $h_j$ the score referred to the horse number $j$.
If the ability of the horse $j$ is unaffected by the rain, then $h_j=0$.
If not, $h_j$ is positive. A higher score will be given to those horses whose ability to run is more affected. In this way, one can set 
$$ h_I(y)
= \tsum_{j=1}^6 h_j\ind_{\{j\}}(y),$$ where $I$ is the information that it's raining and $I_0$ is the initial information about the horses and the weather.
The updated probability that the $j$-th horse wins turns out to be
\[P_I(\{j\})= 
\frac{e^{-h_{j}}  p_j}
{\tsum_{i=1}^6 e^{-h_{i}}  p_i},\]
for $j=1,\dotsc, 6$.
\end{example}
%\end{enumerate}

\section{Discussion}\label{s: discussion}
We have established a framework in which we can update probabilities in the light of general, i.e. non--stochastic, information. 
Given that we cannot connect the information and the outcome of interest via a probability model, we do so through a loss function. 
Minimizing a cumulative loss function involving the information on one side and the  probability distribution on the other, yields the updated probability distribution. 
When the information is stochastic, 
we employ the self information loss function; the solution then reverts to the standard definition of conditional probability. 
%All that is required is to quantify losses connecting the outcome of interest with the given piece of information. And we have shown how this works in a number of examples. We also believe that the concerns about conditional probability and the so--called paradoxes, can be dealt with within our framework. The key is being able to deal with information which arises unexpectedly and for which it is not possible to retrospectively assign a joint probability model to.

%We believe the framework has direct implications for Bayesian inference where ``word of mouth'' information from experts is often obtained. It is interesting to note that our framework allows this information to update $P$ at any point. Usually, the non-stochastic information $I$ comes before the stochastic information given by a random variable $X$, but in our approach we can have $X$ before $I$. The Kullback--Leibler divergence has been proved to be essential for coherence.  

\section*{Acknowledgements}
%We thank two  referees and the Associate Editor and for their helpful comments on an earlier version of the paper which have enabled us to clarify many points. 
This work was partially supported by ESF and Regione Lombardia, Italy (by the grant ``Dote Ricercatori'').

%\begin{comment}
%\nocite{*}										%per includere tutte le voci nella bibliografia
\bibliographystyle{agsm}
\bibliography{Bibliography}
%\end{comment}

\end{document}